\documentclass[runningheads]{llncs}
\usepackage{graphicx}
\usepackage[utf8]{inputenc}
\usepackage{amsgen,amsmath,amstext,amsbsy,amsopn,amssymb}

\RequirePackage{bm}
\usepackage{xcolor}

\allowdisplaybreaks

\newcommand{\ib}{\mathbf{i}}

\newcommand{\qb}{\mathbf{q}}

\newcommand{\vb}{\mathbf{v}}

\newcommand{\bX}{\bm{X}}

\newcommand{\bbR}{\mathbb{R}}

\newcommand{\bbE}{\mathbb{E}}
\newcommand{\bbP}{\mathbb{P}}

\newcommand{\cC}{\mathcal{C}}
\newcommand{\cD}{\mathcal{D}}

\newcommand{\cQ}{\mathcal{Q}}
\newcommand{\cS}{\mathcal{S}}

\newcommand{\cR}{{\mathcal{R}}}

\newcommand{\Var}{\mathrm{Var}}

\begin{document}
\title{PageRank Nibble on the sparse directed stochastic block model\thanks{Supported in part by the NSF RTG award (DMS-2134107) and in part by the NSF-CAREER award (DMS-2141621).}}
\author{Sayan Banerjee  \and
Prabhanka Deka \and
Mariana Olvera-Cravioto \orcidID{0000-0003-3335-759}}
\authorrunning{Banerjee, Deka and Olvera-Cravioto}
\institute{University of North Carolina at Chapel Hill, Chapel Hill NC 27514, USA %\and
%Springer Heidelberg, Tiergartenstr. 17, 69121 Heidelberg, Germany
\email{\{sayan,deka,molvera\}@email.unc.edu}}
\maketitle              % typeset the header of the contribution
\begin{abstract}
We present new results on community recovery based on the PageRank Nibble algorithm on a sparse directed stochastic block model (dSBM). Our results are based on a characterization of the local weak limit of the dSBM and the limiting PageRank distribution. This characterization allows us to estimate the probability of misclassification for any given connection kernel and any given number of seeds (vertices whose community label is known). The fact that PageRank is a local algorithm that can be efficiently computed in both a distributed and asynchronous fashion, makes it an appealing method for identifying members of a given community in very large networks where the identity of some vertices is known. 

\keywords{PageRank Nibble  \and directed stochastic block model \and local weak convergence \and community detection.}
\end{abstract}

\section{Introduction}

Many real-world networks exhibit community structure, where members of the same community are more likely to connect to each other than to members of different communities.  Stochastic block models are frequently used to model random graphs with a community structure, and there are many problems where the goal is to identify the members of a given community, often based only on the graph structure, i.e,, on the vertices and the existing edges among them. A two community symmetric SBM is described by two parameters $\alpha$ and $\beta$, which determine the edge probabilities, with $\alpha$ corresponding to the probability that two members from the same community connect to each other, and $\beta$ to the probability that two members from different communities connect to each other.  In \cite{ref_mns_dense}, the authors work on the semi-sparse regime $\alpha = a \log n /n$ and $\beta = b \log n /n$, where $n$ is the number of vertices in the graph,  and show that the exact recovery of communities is efficiently possible if $| \sqrt{a} - \sqrt{b}| > 2$ and impossible otherwise. When recovery is possible, the authors use spectral methods to get an initial guess of the partition and fine tune it to retrieve the communities. Similar work has been done in the sparse regime, where $\alpha = a/n$ and $\beta = b/n$. In \cite{ref_mns_sparse}, the authors show that recovery is impossible when $(a-b)^2 < 2 (a+b)$. In \cite{ref_massoulie_sparse}, it was proved that recovery is efficiently possible when $(a-b)^2 > 2 (a+b)$ through the use of the spectral properties of a modified adjacency matrix $B$ that counts the number of self avoiding paths of a given length $l$ between two vertices in the graph. Further, the authors of \cite{ref_partialrecovery} show that it is possible to recover a fraction $(1-\gamma)$ of the vertices of community 1 if $a$ and $b$ are sufficiently large and satisfy $(a-b)^2 > K_1 \log(\gamma^{-1}) (a+b)$ for some constant $K_1$ . The clustering methods in \cite{ref_mns_dense,ref_massoulie_sparse,ref_partialrecovery} all rely on finding eigenvectors of the adjacency matrix (or a modified adjacency matrix), which is computationally expensive for large networks. 

%\textcolor{magenta}{Add a few sentences on the top references for community detection on the SBM based only on the graph structure, and include a brief description of the methods. Make sure to include the existing impossibility results}. 

Although the literature on community detection is vast, and there are in fact many methods that work remarkably well, many of those methods become computationally costly for very large networks. In some important cases like the web graph and social media networks, the networks of interest are so large and constantly changing that it becomes difficult to implement some of these methods. Moreover, in many cases, one has more information about the network than just its structure, e.g., vertex attributes that tell us the community to which certain vertices belong to. The question is then whether one can leverage knowledge of such vertices to help identify other members of their community using a computationally efficient method that does not require information about the entire network.  One such problem was studied in \cite{ref_censor}, where the authors consider community detection in a dense (average degree of vertices scale linearly with the size of the network) SBM in which information about the presence or absence of each edge was hidden at random. Here, we will analyze a setting where the labels of some prominent members of the community of interest are known. 

The PageRank Nibble algorithm was introduced in \cite{ref_prnibble} as a modification of the Nibble algorithm described in \cite{ref_localclustering} that uses personalized PageRank. This algorithm provides a cheap method for identifying the members of one community when a number of individuals in that community have been identified.  PageRank based clustering methods were also proposed in \cite{ref_heatsbm} for the two-commmunity SBM, as a special case of a more general method of combining random walk probabilities using a ``discriminant" function. 

The intuition behind PageRank Nibble is that random walks that start with the individuals that are known to belong to the community we seek will tend to visit more often members of that same community. PageRank Nibble works by choosing the personalization parameter of the known individuals, which we refer to as the ``seeds", to be larger than for all other vertices in the network, and then choosing a damping factor $c$ sufficiently far from either 0 or 1. This choice of the personalization values makes the PageRanks of  close neighbors of the seeds to be larger, compared to those of individuals outside the community. Once the ranks produced by PageRank Nibble have been computed, a simple threshold rule can be used to identify the likely members of the community of interest. PageRank based methods can generally be executed quickly due to the availability of fast, distributed algorithms \cite{ref_distributedpr}.

PageRank Nibble on the undirected SBM was studied in \cite{ref_meanfield} under regimes where personalized PageRank (PPR) concentrates around its mean field approximation. The idea proposed there was to use the mean field approximation to identify vertices belonging to the same community as the seeds. In particular, the authors of \cite{ref_meanfield} show that concentration occurs provided the average degrees grow as $a(n) \log n$ for some $a(n) \to \infty$ as $n \to \infty$, and is impossible for the sparse regime where average degrees remain constant as the network size grows.  Our present work focuses on the directed stochastic block model (dSBM) in the sparse regime, and our results are based on the existence of a local weak limit and, therefore, of a limiting PageRank distribution.  Once we have this characterization, we can compute the probability that an individual will be correctly or incorrectly classified, and choose the threshold that minimizes the misclassification probability.

\section{Main Results}

Let $\mathcal{G}_n = G(V_n, E_n)$ be a dSBM on the vertex set $V_n = \{1, \ldots, n \}$ with two communities. To start, each vertex $v \in V_n$ is assigned a latent label $C_v \in \{1, 2\}$ identifying its community. We assume that these labels are unknown to us. Denote by $\cC_1$ and $\cC_2$ the subsets of vertices in communities 1 and 2 respectively. Then, each possible directed edge is sampled independently according to:  
\begin{align*}
    p^{(n)}_{vw} := \bbP((v,w) \in E_n | C_v, C_w ) = 
    \begin{cases} 
        \frac{a}{n} \wedge 1 & \text{if } C_v = C_w\\
        \frac{b}{n} \wedge 1 & \text{if } C_v \neq C_w.
    \end{cases}
\end{align*}
The edge probabilities can be written as $p_{vw}^{(n)} = (n^{-1}\kappa_{C_v, C_w}) \wedge 1$, where $$\kappa = \begin{bmatrix} a & b \\ b & a \end{bmatrix}$$ is called the connection probability kernel for the dSBM. 

For $i=1, 2$, we define $$\pi_i^{(n)} = \frac{1}{n} \sum_{v=1}^{n} 1 (C_v = i) $$ to be the proportion of vertices belonging to each community. We focus specifically on the case where $\pi_1^{(n)} = \pi_2^{(n)} = 1/2$, but the techniques used here can be applied to more general dSBMs.

To describe the setting for our results, start by fixing a constant $0 < s < 1$, and  assume there exists a subset $\cS \subseteq \cC_1$, with $|\cS| = n \pi_1^{(n)} s$, for which the community labels are known. In other words, we assume that we know the identities of a fixed, positive proportion of the vertices in community 1. We refer to the vertices in $\cS$ as the \textit{seeds}. In a real-world social network one can think of the seeds as famous individuals whose community label or affiliation is known or easy to infer.  Given the seed set $\cS$, the goal is to identify the vertices $v \in \mathcal{C}_1 \setminus \cS$, i.e., to recover the remaining members of community 1. 

In order to describe the PageRank Nibble algorithm, we start first with the definition of personalized PageRank. On a directed graph $G = (V,E)$, the PageRank of vertex $v \in V$ is given by:
\begin{align}
r_v = c \sum_{w \in V : (w,v) \in E} \frac{1}{D_w^+} \, r_w  + (1-c) q_v, 
\label{preq}
\end{align}
where $D_w^+$ is the out-degree of vertex $w \in V$, $q_v$ is the personalization value of vertex $v$, and $c \in (0,1)$ is a damping factor. 

PageRank is one of most popular measures of network centrality, due to both its computational efficiency (it can be computed in a distributed and asynchronous way), and its ability to identify {\it relevant} vertices. When $\mathbf{q} = (q_v: v \in V)$ is a probability vector, the PageRank vector $\mathbf{r} = (r_v: v\in V)$ is known to correspond to the stationary distribution of a random walk that, at each time step, chooses, with probability $c$, to follow an outbound edge uniformly chosen at random, or with probability $1-c$, chooses its next destination according to $\mathbf{q}$  (if the current vertex has no outbound edges, the random walk always chooses its next destination according to $\qb$). PageRank is known to rank highly vertices that either have a large in-degree, or that have close inbound neighbors whose PageRanks are very large \cite{ref_olv22corrPR}, hence capturing both popularity and credibility. Since on large networks the PageRank scores will tend to be very small, it is often convenient to work with the scale-free (graph-normalized) PageRank vector $\mathbf{R} = |V| \mathbf{r}$ instead. 

For the two community dSBM $G_n = (V_n, E_n)$ described above, let $Q_v = nq_v$ and define $$\mu_n( B) = \frac{1}{n} \sum_{v=1}^{n} 1((C_v, Q_v) \in B)$$ for any measurable set $B$. We assume that there exists a limiting measure $\mu$ with $\pi_i := \mu(\{i\} \times \bbR_+) > 0$ for $i = 1,2$ such that 
 \begin{align}\mu_n \Rightarrow \mu \label{muconv} \end{align} in probability. Here and in the sequel, $\Rightarrow$ denotes weak convergence. Further, for any measurable A, let 
\begin{align} 
    \sigma_i^{(n)} (A) = \frac{1}{n \pi_i^{(n)}} \sum_{v \in V_n} 1(C_v = i, Q_v \in A), \qquad i = 1,2, \label{Qlimit}
\end{align} 
denote the empirical distribution of $Q_v$ conditionally on $C_v = i$ for $i = 1,2$. Due to assumption \eqref{muconv}, we get the existence of limiting distributions $\sigma_i$, given by 
$$\sigma_i(A) = \frac{\mu(\{i\} \times A)}{\pi_i}, \qquad i = 1,2,$$ such that $\sigma_i^{(n)} \Rightarrow \sigma_i$ in probability as $n \rightarrow \infty$. 

As mentioned in the introduction, our analysis is based on the existence of a local weak limit for the dSBM, and the fact that if we let $I$ be uniformly chosen in $V_n$, independently of $G(V_n, E_n)$, and let $R_I$ denote the scale-free PageRank of vertex $I$, then $R_I$ converges weakly to a random variable $\mathcal{R}$ as $n \to \infty$. In order to characterize the distribution of $\mathcal{R}$, first define $R^{(1)}$ and $R^{(2)}$ to be random variables satisfying
$$\mathbb{P}\left( R^{(i)} \in \cdot \right) = \mathbb{P}\left( \left. R_I \in \cdot \right| C_I = i \right), \qquad i = 1,2.$$
Our first result establishes the weak convergence of $R^{(i)}$ for $i = 1,2$ and characterizes the limiting distributions as the solutions to a system of distributional fixed-point equations.

\begin{theorem} \label{PRequations}
Let $G_n = (V_n, E_n)$ be a sequence of dSBM as described above such that \eqref{muconv} holds. Then, there exist random variables $\mathcal{R}^{(1)}$ and $\mathcal{R}^{(2)}$ such that for any $x \in \mathbb{R}$ that is a point of continuity of the limit,
$$R^{(i)} \Rightarrow \mathcal{R}^{(i)} \qquad \text{and} \qquad \frac{2}{n} \sum_{v\in V_n} 1( R_v \leq x, \, C_v = i) \xrightarrow{P} \bbP\left( \mathcal{R}^{(i)} \leq x \right)  ,$$
as $n \to \infty$, $i = 1,2$. Moreover, the random variables $\mathcal{R}^{(1)}$ and $\mathcal{R}^{(2)}$ satisfy:
\begin{align}
\mathcal{R}^{(1)} \stackrel{d}{=} c \sum_{j=1}^{\mathcal{N}^{(11)}} \frac{\mathcal{R}^{(1)}_j }{\mathcal{D}_j^{(1)}} + c \sum_{j=1}^{\mathcal{N}^{(12)}} \frac{\mathcal{R}^{(2)}_j }{\mathcal{D}_j^{(2)}} + (1-c) \mathcal{Q}^{(1)} \label{r_eq1} \\
\mathcal{R}^{(2)} \stackrel{d}{=} c \sum_{j=1}^{\mathcal{N}^{(21)}} \frac{\mathcal{R}^{(1)}_j }{\mathcal{D}_j^{(1)}} + c \sum_{j=1}^{\mathcal{N}^{(22)}} \frac{\mathcal{R}^{(2)}_j }{\mathcal{D}_j^{(2)}} + (1-c) \mathcal{Q}^{(2)}  \label{r_eq2}
\end{align}
where $\cQ^{(1)}$ and $\cQ^{(2)}$ are random variables distributed according to $\sigma_1$ and $\sigma_2$ respectively, $\mathcal{N}^{(kl)}$ are Poisson random variables with means $\pi_l \kappa_{lk}$, $(\cD_{j}^{(i)} - 1 : j \geq 1)$, $i = 1,2,$ are i.i.d.~sequences of Poisson random variables with mean $\pi_1\kappa_{i1} + \pi_2 \kappa_{i2}$, and $(\mathcal{R}_j^{(i)}: j \geq 1)$, $i=1,2$, are i.i.d.~copies of $\mathcal{R}^{(i)}$, with all random variables independent of each other. 
\end{theorem}
\begin{remark}
    Note that the $(\mathcal{D}_j^{(i)})$ are size-biased Poisson random variables that represent the out-degrees of the inbound neighbors of the explored vertex $I$.    
\end{remark}

The above result holds in more generality for a degree-corrected dSBM with $k$-communities, but for the purposes of this paper, we restrict ourselves to the $k=2$ case. We will only outline a sketch of the proof, and focus our attention instead on the following theorem about the classification of the vertices.

Equations \eqref{r_eq1} and \eqref{r_eq2} are the key behind our classification method. Observe that in the PageRank equations \eqref{preq}, the parameters within our control are the damping factor $c$ and the personalization vector $\mathbf{Q} = (Q_v: v \in V_n)$. If we choose $\mathbf{Q}$ that results in $\cR^{(1)} \geq_\text{s.t.}  \cR^{(2)}$, we can identify vertices in community 1 as the ones having higher PageRank scores. With that in mind, we set $Q_v = 1(v \in \cS)$, choose an appropriate cutoff point $x_0$ (which might depend on c, s and $\kappa$), and classify as a member of community 1 any vertex $v \in V_n$ such that its scale-free PageRank, $R_v$, satisfies $R_v > x_0$.  The algorithm requires that we choose $c$ sufficiently bounded away from both zero and one, since from the random walk interpretation of PageRank, it is clear that we want to give the random surfer time to explore the local neighborhood, while at the same time ensuring that it returns sufficiently often to the seed set. In practice, a popular choice for the damping factor is $c=0.85$. In the context of the dSBM, we have that when $a >> b$, the random surfer ends up spending more time exploring the vertices in community 1, and the probability that it escapes to community 2 before jumping back to the seeds is much smaller. As a result, the stationary distribution ends up putting more mass on the community 1 vertices, and the proportion of misclassified vertices diminishes when $a+b$ is large and $b/a$ is close to zero. We formalize this in the theorem below. Note that Theorem~\ref{PRequations} gives that the miscalssification probabilities satisfy:
$$\quad \mathbb{P}\left( \left. R_v \leq x_0  \right| v \in \mathcal{C}_1 \right) \approx \mathbb{P}\left( \mathcal{R}^{(1)} \leq x_0 \right) \quad \text{and}$$
$$\mathbb{P}\left( \left. R_v > x_0  \right| v \in \mathcal{C}_2 \right) \approx \mathbb{P}\left( \mathcal{R}^{(2)} > x_0 \right) . $$

Our local classification algorithm with input parameters $c$ and $x_0$ is then described as follows:
\begin{enumerate}
    \item Set $Q_v = 1$ if $v \in \cS$, and zero otherwise. 
    \item Fix the damping factor $c \in (0,1)$ and compute the personalized scale-free PageRank vector $\mathbf{R}$.     
    \item For a threshold $x_0$, the estimated members of $\mathcal{C}_1$ are the vertices in the set $\hat{\cC}_1(x_0,c) = \{ v \in V_n: R_v > x_0 \}$.
\end{enumerate}

The theorem below can be used to quantify the damping factor $c$ and the classification threshold $x_0$, and the corollary that follows shows that the proportion of misclassified vertices becomes small with high probability as $n \rightarrow \infty$. 

\begin{theorem} \label{maintheorem}
    Let $G_n = (V_n, E_n)$ be a 2-community dSBM with $$\kappa = \begin{bmatrix}
        a & b\\
        b & a
    \end{bmatrix}$$ and $\pi_1 = \pi_2 = 1/2$. Assume $a,b$ satisfy $ 8b/(a+b) < 1/2$ and ${\rm e}^{-(a+b)/2} < b/4a$. Let $Q_v = 1(v \in \cS)$ for $v \in V_n,$ and take any $c \in (1/2, 1 - 8b/(a+b)]$. Then, for $x_0 = 5s/8$, we have 
    \begin{align}
    \bbP\left(\cR^{(1)} < \frac{5s}{8} \right) &\leq \frac{256 c^2}{(a+b)(1-c^2)} + \frac{64(1-c)(1-s)}{(1+c)s} \label{c1misclassification},\\
        \bbP\left(\cR^{(2)} > \frac{5s}{8} \right) &\leq \frac{256 c^2}{(a+b)(1-c^2)}\left(1 + \frac{(1-c)(1-s)}{2(1+c)s} \right). \label{c2misclassification}
    \end{align}
    
    %\begin{align}
     %   \bbP \left(\cR^{(1)} \leq x_0 \right) &\leq 384 \frac{c^2}{(p+q)s^2} + 4096 \frac{1-s}{s} \left(\frac{q}{p+q}\right)^2 \label{c1misclassification}\\
      %  \bbP \left(\cR^{(2)} > x_0 \right) &\leq 384 \frac{c^2}{(p+q)s^2}. \label{c2misclassification}
    %\end{align}
\end{theorem}

Naturally, the misclassification errors get smaller as $s$ increases, i.e., as more members of community 1 are known. Also, we get better bounds for the misclassification errors when $a+b$ is large (strong connectivity within a community) and $b/(a+b)$ is small (equivalently, $a/(a+b)$ close to one), i.e., when the network is strongly assortative. 

Note that the assumptions in Theorem \ref{maintheorem} do not involve $s$ (proportion of seeds). As the proof indicates, our classification errors involve Chebychev bounds which crucially depend on: (i) the mean PageRank scores of the two communities being sufficiently different, and (ii) the ratio of the variance of the PageRank scores of vertices in each community to the square of the mean community PageRank being small. By Lemma \ref{firstmomentbound} below, the ratio of the mean community PageRank scores is independent of $s$ and hence their separation required by (i) is ensured by conditions involving $a,b$ but not $s$. Moreover, as seen in Lemma \ref{secondmomentbound}, the scaled fluctuations in (ii) depend more significantly on the `sparsity' of the underlying network, quantified by $a+b$ (expected total degree of a vertex), than $s$. Thus, the dependence on $s$ arises mainly through the choice of the threshold $x_0$ in our classification algorithm (see Corollary \ref{cor}).

As a direct corollary to Theorem \ref{maintheorem}, we have
\begin{corollary}\label{cor}
Let $x_0 = 5s/8$, $c \in (1/2, 1 - 8b/(a+b)]$, $$\delta_1 = \frac{256 c^2}{(a+b)(1-c^2)} + \frac{64(1-c)(1-s)}{(1+c)s} $$ and $$\delta_2 = \frac{256 c^2}{(a+b)(1-c^2)}\left(1 + \frac{(1-c)(1-s)}{2(1+c)s} \right).$$ Then, under the hypothesis of Theorem 2, for $\delta = \delta_1 + \delta_2$ and any $\epsilon > 0$, we have 
$$\lim_{n \rightarrow \infty } \bbP \left( | \cC_1 \triangle \hat{\cC}_1(x_0,c)| > \frac{(\delta + \epsilon) n}{2}\right) = 0.$$
\end{corollary}
\begin{proof}
    For notational convenience, we drop the dependence of $\hat{\mathcal{C}}_1$ on $x_0$ and $c$. Observe that $| \cC_1 \triangle \hat{\cC}_1| = |\cC_1 \backslash \hat{\cC}_1| + | \hat{\cC}_1 \cap \cC_2|$, and we have $\cC_1 \backslash \hat{\cC}_1 = \{v \in \cC_1 : R_v < 5s/8 \}$ and $\hat{\cC}_1 \cap \cC_2 = \{ v \in \cC_2 : R_v > 5s/8 \}$. So we get that for $x_0 = 5s/8$,
    $$\bbP \left(| \cC_1 \triangle \hat{\cC}_1| > \frac{(\delta+\epsilon) n}{2} \right) = \bbP \left( \frac{2}{n}\sum_{v \in \cC_1} 1(R_v < x_0) + \frac{2}{n} \sum_{v \in \cC_2} 1(R_v > x_0) > \delta +\epsilon \right).$$ Then the result follows since 
    \begin{align*}
       & \frac{2}{n}\sum_{v \in \cC_1} 1(R_v < x_0) +  \frac{2}{n} \sum_{v \in \cC_2} 1(R_v > x_0) \\
       &\xrightarrow{P} \bbP(\cR^{(1)} < x_0)+ \bbP(\cR^{(2)} > x_0) = \delta
    \end{align*}
    as $n \rightarrow \infty$.
\end{proof}

\begin{remark}\label{opco}
Our proof of Theorem \ref{maintheorem} uses Chebyshev’s inequalities based on mean and variance bounds for the limiting (scale-free) personalized PageRank distribution obtained from the distributional fixed-point equations in Theorem \ref{PRequations}. The choice of $x_0$ above is rather ad hoc and mainly for simplicity of the associated misclassification error bounds. One can check that the choice of $x_0$ which minimizes the sum of the Chebyshev error bounds is given by $x_0^* = (r_1 v_2^{1/3} + r_2v_1^{1/3})/(v_1^{1/3} + v_2^{1/3})$, where $r_1,r_2$ are the expected limiting PageRank values obtained in Lemma \ref{firstmomentbound} and $v_1,v_2$ are the corresponding variances obtained in Lemma \ref{secondmomentbound}. Further, $x_0 = 5s/8$ is independent of the kernel parameters $a$ and $b$, which are often unknown in practice. Moreover, although the range of $c$ depends on $a,b$, the results above hold for any $c$ in the given range. Hence, in practice, when $a,b$ are not known, then any $c > 1/2$ which is not too close to one should work provided the network is not too sparse ($b/(a+b)$ is sufficiently small). 
\end{remark}

%\section{Related works}

%Graph partitioning and clustering using PageRank has become a key tool for networks with a large number of nodes due to the local nature of PageRank and the relative ease and efficiency of solving the linear equations associated with it. Chung \cite{ref_prheatkernel} explores PageRank as a heat kernel, establishing a local Cheeger inequality and proposes a partition algorithm by finding local cuts. Kloumann, Ugander and Kleinberg \cite{ref_heatsbm} gives results on community recovery for a 2-community SBM in the dense regime using pesonalized PageRank. Avrachenkov, Kadavankandy and Litvak \cite{ref_meanfield} study personalized PageRank on Erd\H{o}s-R\`enyi graphs containing a denser, planted Erd\H{o}s-R\`enyi subgraph and establishes regimes where PPR concentrates or does not concentrate around the mean field approximation. In our model, we do not have concentration, which explains why the error bounds in equations $\eqref{c1misclassification}$ and $\eqref{c2misclassification}$ does not vanish as $n \rightarrow \infty.$ Gulikers, Lelarge and Massouli\`e \cite{ref_dcsbm} further proved that for a 2-community SBM where $(p-q)^2 \leq 2(p+q)$, it is impossible to estimate the communities in the absence of any other information. The fact that we know the identity of a positive proportion of seeds in community 1 is crucial to our results.

\section{Proofs}

As mentioned earlier, Theorem~\ref{PRequations} holds in considerably more generality than the one stated here, so we will only provide a sketch of the proof that suffices for the simpler dSBM considered here. The proof of Theorem~\ref{maintheorem} is given later in the section.

\begin{proof}{Theorem~\ref{PRequations} (Sketch).}
The proof consists of three main steps. 
\begin{enumerate}
\item \textbf{Establish the local weak convergence of the dSBM:} For the 2-community dSBM considered here, one can modify the coupling in \cite{ref_dcsbm} (which works for an undirected SBM) to the exploration of the in-component of a uniformly chosen vertex. The coupled graph is a 2-type Galton-Watson process, with the two types corresponding to the two communities in the dSBM, and all edges directed from offspring to parent. The number of offspring of type $j$ that a node of type $i$ has is a Poisson random variable with mean $m_{ij}^- = \pi_j \kappa_{ji}$ for $j=1,2$. For each node $\ib $ in the coupled tree, denote by $C_{\ib}$ its type, and assign it a mark $\bX_{\ib} = ( \mathcal{D}_{\ib}, Q_{\ib})$, where $(\mathcal{D}_{\ib} - 1 )| C_{\ib} = j$ is a Poisson random variable with mean $m_j^+ = \pi_1 \kappa_{j1} + \pi_2 \kappa_{j2} $, and $Q_{\ib} | C_{\ib} = j$ has distribution $\sigma_j$ as defined in \eqref{Qlimit}. The construction of the coupling follows a two step exploration process similar to the one done for inhomogeneous random digraphs in \cite{ref_prinhomogeneous}. First the outbound edges of a vertex are explored, followed by the exploration of its inbound neighbors, assigning marks to a vertex once we finish exploring both its inbound and outbound one-step neighbors. This establishes the local weak convergence in probability of the dSBM to the 2-type Galton-Watson process. 

\item \textbf{Establish the local weak convergence of PageRank:} Once we have the local weak convergence of the dSBM, let $\mathcal{R}^*$ denote the personalized PageRank of the root node of the 2-type Galton-Watson process in the coupling. The local weak convergence in probability of the PageRanks on the dSBM to $\mathcal{R}^*$, i.e., 
$$\frac{1}{n} \sum_{v \in V_n} 1( R_v \leq x) \xrightarrow{P} \bbP (\mathcal{R}^* \leq x)$$
as $n \to \infty$, follows from Theorem~2.1 in \cite{ref_weakconvpr}. Note that the random variables $\mathcal{R}^{(1)}$ and $\mathcal{R}^{(2)}$ correspond to the conditional laws of $\mathcal{R}^*$ given that the root has type 1 or type 2, respectively. And since the two communities are assumed to have the same size, the probability that the root has type 1 is $1/2$, hence,
$$\frac{1}{n} \sum_{v \in V_n} 1( R_v \leq x, \, C_v = i) \xrightarrow{P} \bbP(\mathcal{R}^{(i)} \leq x ) \frac{1}{2},$$
as $n \to \infty$. The weak convergence result follows from the bounded convergence theorem.

\item \textbf{Derive the distributional fixed point equations:} If the nodes in the first generation of the 2-type Galton-Watson process are labeled $1 \leq j \leq \mathcal{N}$, where $\mathcal{N}$ denotes the number of offspring of the root node, then 
$$\mathcal{R}^* = c \sum_{j=1}^{\mathcal{N}} \frac{\mathcal{R}_j}{\mathcal{D}_j} + (1-c) \mathcal{Q},$$
where $\mathcal{Q}$ denotes the personalization value of the root, $(\mathcal{D}_j: j \geq 1)$ correspond to the out-degrees of the offspring, and the $(\mathcal{R}_j: j \geq 1)$ correspond to their PageRanks. Conditioning on the type of the root, as well as on the types of its offspring, gives the two distributional fixed-point equations \eqref{r_eq1} and \eqref{r_eq2}. In particular, conditionally on the root having type $i$,  $\mathcal{N}^{(ik)}$ corresponds to the number of offspring of type $k$, $\mathcal{Q}^{(i)}$ has distribution $\sigma_i$, and $\mathcal{D}_1^{(k)}$ and $\mathcal{R}_1^{(k)}$ are independent random variables having the distribution of $\mathcal{D}_1$ and $\mathcal{R}_1$ conditionally on node $1$ having type $k$.
\end{enumerate}
\end{proof}

We prove Theorem $\ref{maintheorem}$ through the second moment method. First we prove the following lemmas establishing bounds on the mean and variance of $\cR^{(1)}$ and $\cR^{(2)}$.

\begin{lemma} \label{firstmomentbound}
Let $r_i = \bbE \left[ \cR^{(i)} \right]$, $\lambda = 1 - {\rm e}^{-(a+b)/2}$ and $$\Delta = \left(1 - \frac{c \lambda a}{a+b} \right)^2 - \left( \frac{c \lambda b}{a+b}\right)^2 .$$ Then, we have 
\begin{align}
    r_1 &= \frac{\left( 1 - \frac{c \lambda  a}{a+b}\right)s(1-c)}{\Delta} \label{r1mean}\\
    r_2 &= \frac{ \left( \frac{c \lambda b}{a+b}\right)s(1-c)}{\Delta} \label{r2mean}.
\end{align}
Further, if $1 - \lambda = {\rm e}^{-(a+b)/2} \leq b/4a$ and $c >1/2$, we have the bounds
\begin{align}
    r_1 &\geq s \left(1 - \frac{2b}{(1-c) (a+b)}\right) \label{r1bound},\\
    r_2 &\leq \frac{s}{2} \label{r2bound}.
\end{align}
\end{lemma}
\begin{proof}
Recall the distributional equations satisfied by $\cR^{(1)}$ and $\cR^{(2)}$ from Theorem~\ref{PRequations}. Taking expectation on both sides gives us
\begin{align*}
    \bbE[\cR^{(1)}] &= c \bbE \left[  \sum_{j=1}^{\mathcal{N}^{(11)}}\frac{\cR_j^{(1)}}{\cD_j^{(1)}} + \sum_{j=1}^{\mathcal{N}^{(12)}}\frac{\cR_j^{(2)} }{\cD_j^{(2)}} \right]  + (1-c) \bbE [\cQ^{(1)}],\\
    \bbE[\cR^{(2)}] &= c \bbE \left[ \sum_{j=1}^{\mathcal{N}^{(21)}}\frac{\cR_j^{(1)}}{\cD_j^{(1)}} + \sum_{j=1}^{\mathcal{N}^{(22)}} \frac{\cR_j^{(2)}}{\cD_j^{(2)}} \right]   + (1-c) \bbE[\cQ^{(2)}].
\end{align*}
First, note that with our choice of $\mathcal{Q}$, $\bbE[\cQ^{(1)}] = s$  and $\bbE[\cQ^{(2)}] = 0$. Further $(\cR_j^{(i)}, \cD_j^{(i)})_{j \geq 1}$ (resp. $(\cR_j^{(i)}, \cD_j^{(i)})_{j \geq 1}$) are independent of $\mathcal{N}^{(1i)}$ (resp. $\mathcal{N}^{(2i)}$), and of each other, for $i = 1, 2$. So the above expressions can be simplified to 
\begin{align*}
    r_1 &= c \left( \bbE[ \mathcal{N}^{(11)}] \bbE \left[\frac{1}{\cD^{(1)}} \right] r_1 + \bbE[ \mathcal{N}^{(12)} ] \bbE \left[\frac{1}{\cD^{(2)}} \right] r_2 \right) + (1-c)s,\\
    r_2 &= c \left( \bbE[ \mathcal{N}^{(21)}] \bbE \left[\frac{1}{\cD^{(1)}} \right] r_1 + \bbE[ \mathcal{N}^{(22)}] \bbE \left[\frac{1}{\cD^{(2)}} \right] r_2 \right),
\end{align*}
where $\mathcal{N}^{(ij)}$ and $(\cD^{(i)} - 1)$ are Poisson random variables with means as described in Theorem~\ref{PRequations}. Therefore we can further reduce the equations to
\begin{align*}
    r_1 &= c\left( \frac{a}{2} \cdot \frac{(1 - {\rm e}^{-(a+b)/2})}{(a+b)/2} \cdot r_1 + \frac{b}{2} \cdot \frac{(1 - {\rm e}^{-(a+b)/2})}{(a+b)/2} \cdot r_2 \right) + (1-c)s,\\
    r_2 &= c \left(\frac{b}{2} \cdot \frac{(1 - {\rm e}^{-(a+b)/2})}{(a+b)/2} \cdot r_1 + \frac{a}{2} \cdot \frac{(1 - {\rm e}^{-(a+b)/2})}{(a+b)/2} \cdot r_2 \right),
\end{align*}
or in matrix form, and after substituting $\lambda = (1 - {\rm e}^{-(a+b)/2})$, 
\begin{align}
    \begin{bmatrix}
        1 -  \frac{ca \lambda}{a+b} & -\frac{cb\lambda}{a+b}\\
        -\frac{cb\lambda}{a+b} & 1 - \frac{ca\lambda}{a+b}
    \end{bmatrix}
    \begin{bmatrix}
        r_1 \\ r_2
    \end{bmatrix}
    = \begin{bmatrix}
        (1-c)s \\ 0
    \end{bmatrix}. \label{matrixform}
\end{align}
Solving \eqref{matrixform}, we get 
\begin{align*}
\left[ \begin{matrix} r_1 \\ r_2 \end{matrix} \right] &= \frac{1}{\Delta} \left[ \begin{matrix} \left( 1 - c\lambda a/(a+b) \right) s(1-c) \\ c\lambda b s(1-c)/(a+b) \end{matrix} \right],
\end{align*}
where  $$\Delta = \left(1 - \frac{c \lambda a}{a+b} \right)^2 - \left( \frac{c \lambda b}{a+b}\right)^2 $$ as required. Note that since $c \lambda < 1$, we have $\Delta > 0$, and so the above quantities are well defined. From here, the bound for $r_2$ is a straightforward calculation.
\begin{align*}
    r_2 &= \frac{\frac{c \lambda b}{a+b}  (1-c)s}{\Delta} \leq \frac{\frac{b}{a+b} (1-c)s}{(1 - c \lambda)\left(1 - c \lambda \frac{a-b}{a+b} \right)} \leq \frac{sb}{(a+b)} \frac{1}{\left( 1 - \frac{a-b}{a+b}\right)} = \frac{s}{2}.
\end{align*}
To get the bound for $r_1$, we proceed as follows
\begin{align*}
    r_1 &= \frac{\left( 1 - \frac{c \lambda a}{a+b}\right) (1-c)s}{\Delta}\\
    &\geq \frac{s(1-c)}{1 - \frac{c \lambda a}{a+b}} = \frac{s (1-c)}{\frac{a+b}{a+b} - \frac{c\lambda a}{a+b}} = \frac{s (1-c) }{\frac{b}{a+b} + \frac{a}{a+b}\left( 1 - \lambda + \lambda (1-c) \right)}\\
    &\geq \frac{s (1-c) }{\frac{b}{a+b} + \frac{a}{a+b} \left({\rm e}^{-(a+b)/2} + (1-c) \right)} = \frac{s (1-c) }{(1-c) + \frac{cb}{a+b} + \frac{a}{a+b}{\rm e}^{-(a+b)/2}}\\
    &\geq  \frac{s(1-c)}{1-c + \frac{2bc}{a+b}} \geq s \left( 1 - \frac{2b}{(1-c)(a+b)} \right),
    \end{align*}
where for the last inequality we used fact that ${\rm e}^{-(a+b)/2}a/(a+b) \leq b/4(a+b) \leq cb/(a+b)$ due to our assumptions on $\lambda$ and $c$, and $1 - x^2 \leq 1$ for all $x \in \bbR$. This completes the proof.
\end{proof}
The next lemma provides a result for the variances. 

\begin{lemma}\label{secondmomentbound} 
Define $v_i = \Var(\mathcal{R}^{(i)})$ for $i=1,2$, then, if we let $\mathbf{v} = (v_1, v_2)'$, and $\mathbf{r^2} = (r_1^2, r_2^2)'$, then
$$\mathbf{v}  = \frac{1}{2(1-g_1)(1-g_2)} \left( K \mathbf{r^2} + (1-c)^2 s(1-s) \mathbf{k} \right),$$
where $g_1 = c^2 \bbE[1/(\cD^{(1)})^2] (a-b)/2$, $g_2 = c^2 \bbE[1/(\cD^{(1)})^2] (a+b)/2$, 
$$K  = \left[ \begin{matrix} g_1 + g_2 - 2g_1 g_2 , & g_2  - g_1  \\ g_2  - g_1 , & g_1 + g_2 - 2g_1 g_2\end{matrix} \right], \quad \text{and} \quad \mathbf{k} =  \left[ \begin{matrix} 2- g_1-g_2  \\ g_2- g_1  \end{matrix} \right] .$$
Furthermore,
\begin{align*}
    v_1 &\leq \frac{4c^2s^2}{(a+b)(1-c^2)} + \frac{1-c}{1+c}s(1-s), \\
    v_2 &\leq \frac{4c^2s^2}{(a+b)(1-c^2)}\left(1 + \frac{(1-c)(1-s)}{2s(1+c)}\right).
\end{align*}
    %\begin{align}
      % v_1 &\leq 6 \frac{c^2 }{p+q} + (1-c)^2 s (1-s) \label{r1variance}\\
      % v_2 &\leq 6 \frac{c^2}{p+q}, \label{r2variance}
    %\end{align}
\end{lemma}
\begin{proof}
    To calculate the variance of $\cR^{(1)}$ and $\cR^{(2)}$, we will rely on the law of total variances, i.e., for any two random variables $X$ and $Y$, $$\Var(X) =  \Var[\bbE(X|Y)] + \bbE [ \Var(X|Y)].$$ Applying this to equation \eqref{r_eq1}, along with the fact that $r_i < 1$ for $i = 1, 2$, we get the following bound for $\Var(\cR^{(1)})$:
    \begin{align*}
        \Var(\cR^{(1)}) =& \ c^2 \Var \left( r_1 \mathcal{N}^{(11)} \bbE\left[ \frac{1}{\cD^{(1)}} \right] + r_2 \mathcal{N}^{(12)} \bbE\left[ \frac{1}{\cD^{(2)}}\right] \right) \\  &+ c^2 \bbE \left[ \mathcal{N}^{(11)} \Var\left( \frac{\cR^{(1)}}{\cD^{(1)}}\right) + \mathcal{N}^{(12)} \Var\left( \frac{\cR^{(2)}}{\cD^{(2)}}\right) \right] + (1-c)^2 \Var(\cQ^{(1)}) .
    \end{align*}
Now use the observation that $\Var(\mathcal{N}^{(11)}) = \bbE[\mathcal{N}^{(11)}] = a/2$, $\Var(\mathcal{N}^{(12)}) = \bbE[ \mathcal{N}^{(12)}] = b/2$, 
%$\bbE[ (\cD^{(1)})^{-1}] = \bbE[(\cD^{(2)})^{-1}] = 2(1- e^{-(p+q)/2})/(p+q) = 2\lambda/(p+q)$, 
and $\Var(\mathcal{Q}^{(1)}) = s(1-s)$, to obtain that for $v_i = \Var(\mathcal{R}^{(i)})$,
\begin{align*}
v_1 &= c^2 \left( r_1^2 \cdot \frac{a}{2} \cdot \left(\bbE\left[\frac{1}{\cD^{(1)}} \right]\right)^2  + r_2^2 \cdot  \frac{b}{2} \cdot \left(\bbE\left[\frac{1}{\cD^{(2)}} \right]\right)^2  \right) \\
&\hspace{5mm} + c^2 \left( \frac{a}{2} \Var\left( \frac{\cR^{(1)}}{\cD^{(1)}}\right) + \frac{b}{2} \Var\left( \frac{\cR^{(2)}}{\cD^{(2)}}\right)  \right) + (1-c)^2 s(1-s) \\
&= c^2 \left( r_1^2 \cdot \frac{a}{2} \cdot \left(\bbE\left[\frac{1}{\cD^{(1)}} \right]\right)^2  + r_2^2 \cdot  \frac{b}{2} \cdot \left(\bbE\left[\frac{1}{\cD^{(2)}} \right]\right)^2  \right) + (1-c)^2 s(1-s) \\
&\hspace{5mm} + \frac{c^2 a}{2} \left(  \Var\left(  \frac{1}{\cD^{(1)}}  r_1 \right) +  \bbE\left[ \frac{1}{(\cD^{(1)})^2 } \Var(\mathcal{R}^{(1)})  \right]  \right) \\
&\hspace{5mm} + \frac{c^2 b}{2} \left(  \Var\left(  \frac{1}{\cD^{(2)}}  r_2 \right) +  \bbE\left[ \frac{1}{(\cD^{(2)})^2 } \Var(\mathcal{R}^{(2)})  \right]  \right) \\
&= c^2 \left( r_1^2 \cdot \frac{a}{2} \cdot \left(\bbE\left[\frac{1}{\cD^{(1)}} \right]\right)^2  + r_2^2 \cdot  \frac{b}{2} \cdot \left(\bbE\left[\frac{1}{\cD^{(2)}} \right]\right)^2  \right) + (1-c)^2 s(1-s) \\
&\hspace{5mm} + \frac{c^2 a}{2} \left( r_1^2 \Var\left(  \frac{1}{\cD^{(1)}} \right) +  \bbE\left[ \frac{1}{(\cD^{(1)})^2 }   \right] v_1  \right) \\
&\hspace{5mm} + \frac{c^2 b}{2} \left( r_2^2 \Var\left(  \frac{1}{\cD^{(2)}} \right) +  \bbE\left[ \frac{1}{(\cD^{(2)})^2 } \right] v_2  \right) \\
&= (1-c)^2 s(1-s)    + \frac{c^2 a}{2}   \bbE\left[ \frac{1}{(\cD^{(1)})^2 }   \right] (r_1^2 + v_1)   + \frac{c^2 b}{2}  \bbE\left[ \frac{1}{(\cD^{(2)})^2 } \right] (r_2^2 + v_2 ).
\end{align*}
Similarly, using $\mathcal{Q}^{(2)} \equiv 0$ and 
\begin{align*}
v_2 &= c^2 \Var \left( r_1 \mathcal{N}^{(21)} \bbE\left[ \frac{1}{\cD^{(1)}} \right] + r_2 \mathcal{N}^{(22)} \bbE\left[ \frac{1}{\cD^{(2)}}\right] \right) \\  
&\hspace{5mm} + c^2 \bbE \left[ \mathcal{N}^{(21)} \Var\left( \frac{\cR^{(1)}}{\cD^{(1)}}\right) + \mathcal{N}^{(22)} \Var\left( \frac{\cR^{(2)}}{\cD^{(2)}}\right) \right] + (1-c)^2 \Var(\cQ^{(2)}) \\
%&= c^2 \left( r_1^2 \cdot \frac{q}{2} \cdot \frac{4\lambda^2}{(p+q)^2}  + r_2^2 \cdot  \frac{p}{2} \cdot \frac{4\lambda^2}{(p+q)^2}  \right) \\
%&\hspace{5mm} + c^2 \left( \frac{q}{2} \Var\left( \frac{\cR^{(1)}}{\cD^{(1)}}\right) + \frac{p}{2} \Var\left( \frac{\cR^{(2)}}{\cD^{(2)}}\right)  \right) \\
&= \frac{c^2 b}{2}   \bbE\left[ \frac{1}{(\cD^{(1)})^2 }   \right] (r_1^2 + v_1)   + \frac{c^2 a}{2}  \bbE\left[ \frac{1}{(\cD^{(2)})^2 } \right] (r_2^2 + v_2 ).
\end{align*}
Writing the above in matrix notation we obtain for $\mathbf{v} = (v_1, v_2)'$ and $\mathbf{r^2} = (r_1^2, r_2^2)'$,
\begin{align*}
\mathbf{v} &= c^2 M (\mathbf{v} + \mathbf{r^2}) + \mathbf{h},
\end{align*}
where (note that $\cD^{(1)} \stackrel{d}{=} \cD^{(2)}$), 
$$M =  \frac{\bbE\left[ 1/(\cD^{(1)})^2 \right]}{2} \left[ \begin{matrix} a   & b  \\
b   & a  \end{matrix} \right] \qquad \text{and} \qquad \mathbf{h} = \left[ \begin{matrix} (1-c)^2 s(1-s) \\ 0 \end{matrix} \right].$$
Moreover, use the observation that
$$M = B A B, \quad A = \frac{\bbE[1/(\cD^{(1)})^2]}{2} \left[ \begin{matrix} (a-b) & 0 \\ 0 & (a+b) \end{matrix} \right], \quad B = \frac{1}{\sqrt{2}} \left[ \begin{matrix} -1 & 1 \\ 1 & 1 \end{matrix} \right],$$
so the maximum eigenvalue of $M$ is $\bbE[1/(\cD^{(1)})^2] \bbE[\cD^{(1)}-1]$. Since for a Poisson random variable $N$ with mean $\mu$ we have that
\begin{equation}\label{bf}
\bbE[1/(N+1)^2] E[N] = \sum_{n=1}^\infty \frac{e^{-\mu} \mu^n}{n \cdot n!} = \bbE\left[ \frac{1}{N \vee 1}\right] - e^{-\mu} \leq 1,
\end{equation}
then the matrix $I - c^2 M$ is invertible, and we obtain
\begin{align*}
\mathbf{v} &= (I - c^2 M)^{-1} ( c^2 M \mathbf{r^2} + \mathbf{b}) = \sum_{k=0}^\infty c^{2k} M^k (c^2 M \mathbf{r^2} + \mathbf{b}) \\%&= \sum_{k=0}^\infty c^{2(k+1)} B A^{k+1} B \mathbf{r^2} + \sum_{k=0}^\infty c^{2k} B A^{2k}B \mathbf{b} \\
&= B \left[ \begin{matrix} \frac{c^2 A_{11}}{1- c^2 A_{11}} & 0 \\ 0 & \frac{c^2 A_{22}}{1- c^2 A_{22}} \end{matrix} \right] B \mathbf{r^2} + B \left[ \begin{matrix} \frac{1}{1-c^2 A_{11}} & 0 \\ 0 & \frac{1}{1-c^2 A_{22}} \end{matrix} \right] B \mathbf{b}.
\end{align*}
Setting $g_i = c^2 A_{ii}$ for $i = 1,2$, and computing the product of matrices gives:
$$\mathbf{v}  = \frac{1}{2(1-g_1)(1-g_2)} \left( K \mathbf{r^2} + (1-c)^2 s(1-s) \mathbf{k} \right),$$
for $K$ and $\mathbf{k}$ defined in the statement of the lemma.

Further, if we let $\Delta_2 = 2(1-g_1)(1-g_2)$ and expand the above equation, we obtain 
\begin{align*}
    \vb = \frac{1}{\Delta_2} \left(\begin{bmatrix}
        (g_1 + g_2 - 2g_1g_2)r_1^2 + (-g_1 + g_2) r_2^2 \\  (-g_1 + g_2) r_1^2 + (g_1 + g_2 - 2g_1g_2)r_2^2
     \end{bmatrix} + (1-c)^2s(1-s)\begin{bmatrix}
         (2 - (g_1+g_2)) \\ (-g_1 + g_2)
     \end{bmatrix} \right).
\end{align*}
From equations $\eqref{r1mean}$ and $\eqref{r2mean}$ we also get that $r_i \leq s$ for $i=1,2,$ so we can reduce this to
\begin{align*}
    \vb \leq \frac{1}{\Delta_2} \begin{bmatrix}
        2g_2(1-g_1)s^2 + (2 - (g_1+g_2))(1-c)^2s(1-s)\\
        2g_2(1-g_1)s^2 + (-g_1 + g_2)(1-c)^2s(1-s)
    \end{bmatrix}.
\end{align*}
Plugging in $\Delta_2 = 2(1-g_1)(1-g_2)$, and noting that $g_2 \geq g_1 $, we get
\begin{align*}
    v_1 &\leq \frac{g_2 s^2}{1 - g_2} + \frac{1}{2}\left( \frac{1}{1- g_2} + \frac{1}{1 - g_1}\right)(1-c)^2s(1-s)\\
    &\leq \frac{g_2 s^2}{1 - g_2} + \frac{1}{1 - g_2}(1-c)^2s(1-s),
\end{align*}
and
\begin{align*}
    v_2 &\leq \frac{g_2 s^2}{1-g_2} + \frac{1}{2}\left( \frac{1}{1- g_2} - \frac{1}{1 - g_1} \right)(1-c)^2s(1-s) \\
    &\leq \frac{g_2 s^2}{1-g_2} + \frac{g_2}{2(1-g_1)(1-g_2)}(1-c)^2s(1-s).
\end{align*}
Finally, using $\bbE[1 / (\cD^{(1)})^2 ] \leq 8/(a+b)^2$ and \eqref{bf}, we have $g_2 \leq \min\{c^2,4c^2/(a+b)\}$, and so
\begin{align*}
    v_1 &\leq \frac{4c^2s^2}{(a+b)(1-c^2)} + \frac{1-c}{1+c}s(1-s), \\
    v_2 &\leq \frac{4c^2s^2}{(a+b)(1-c^2)}\left(1 + \frac{(1-c)(1-s)}{2s(1+c)}\right).
\end{align*}
%\textcolor{magenta}{Need to derive "nicer" upper bounds still.}

\end{proof}

We are now ready to prove Theorem \ref{maintheorem}.
\begin{proof}[Proof of Theorem \ref{maintheorem}] For any $z > 0$, Chebyshev's inequality gives
\begin{align}
    \notag \bbP(\cR^{(1)} \leq r_1 - z) &= \bbP(\cR^{(1)} - r_1 \leq -z) \\
    \notag &\leq \frac{v_1}{z^2} \\
    &\leq \frac{1}{z^2}\left(\frac{4c^2s^2}{(a+b)(1-c^2)} + \frac{1-c}{1+c}s(1-s) \right) \label{r1conc}.
\end{align}
A similar application of Chebyshev's inequality for any $w > 0$ with $\cR^{(2)}$ gives 
\begin{align}
    \bbP\left(\cR^{(2)} > \frac{s}{2} + w \right) &\leq \bbP(\cR^{(2)} > r_2 + w) \leq \frac{v_2}{w^2} \notag\\
    &\leq \frac{1}{w^2}  \frac{4c^2s^2}{(a+b)(1-c^2)}\left(1 + \frac{(1-c)(1-s)}{2s(1+c)} \right) \label{r2conc},
\end{align}    
where the first inequality follows from equation \eqref{r2bound}.
Choosing $c \in (1/2, 1 - 8b/(a+b)]$ results in $r_1 \geq 3s/4$, so choosing $z = w = s/8$ and plugging into the bounds from equations \eqref{r1conc} and \eqref{r2conc} gives
\begin{align*}
    \bbP\left(\cR^{(1)} < \frac{5s}{8} \right) &\leq \frac{256 c^2}{(a+b)(1-c^2)} + \frac{64(1-c)}{1+c}\cdot \frac{1-s}{s}, \\
        \bbP\left(\cR^{(2)} > \frac{5s}{8} \right) &\leq \frac{256 c^2}{(a+b)(1-c^2)}\left(1 + \frac{(1-c)(1-s)}{2s(1+c)} \right). 
\end{align*}
\end{proof} 

\section{Results from simulations}
We illustrate the algorithm with some simulation experiments. First, we calculated the personalized PageRank scores for a 2-community dSBM with $n=20000$ vertices,  $a=150$, $b=10$, $s = .2$ and $c = .85$. The plot shows a clear separation of the PPR scores of the seeds, the rest of community 1 and the vertices in community 2.
\begin{figure}[ht]
\centering
\includegraphics[scale = 0.25]{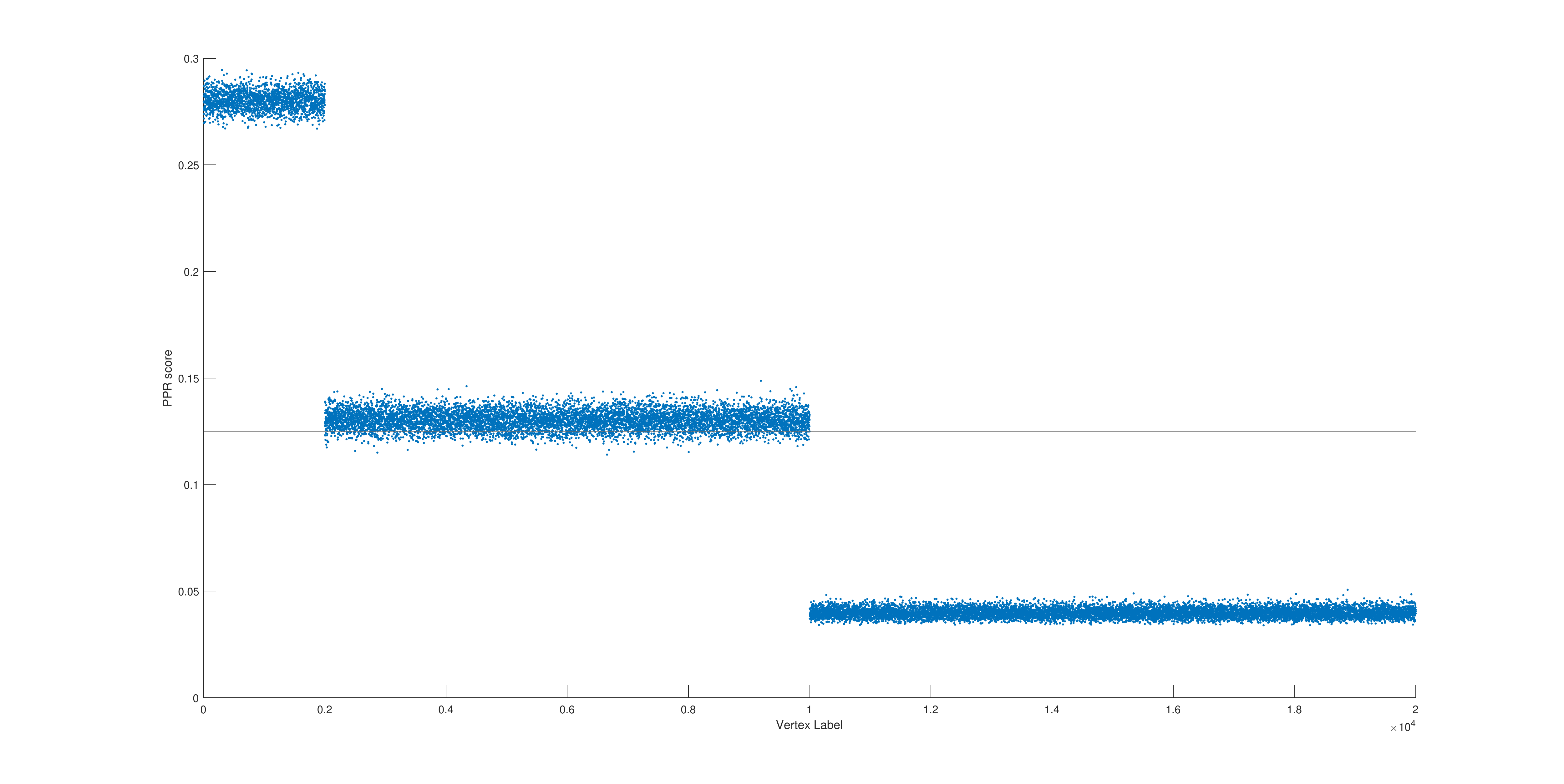}
\caption{A plot of the personalized PageRank scores for a 2-community dSBM with $a=150$, $b=10$, $n=20000$, $s=0.2$, and $c = .85$. The first 2000 vertices are the seeds, vertices 2001-10000 are the rest of community 1, and vertices 10001-20000 are community 2. The horizontal black line is our cutoff level $5s/8$. Proportion of misclassified community 1 vertices is 0.0935.}
\label{fig:prplot}
\end{figure}

We also investigated the role of the damping factor $c$ and the best way to choose it. One natural way of doing so is to find the value of $c$ that maximizes the difference between the mean PPR scores for the two communities. Note that $r_1 - r_2$ is strictly monotone in $c$, but if we let $\hat{r}_1$ to be the mean of the non-seeded members of community 1, we see in Fig. \ref{fig:diffofmeans} that $\hat{r}_1 - r_2$ is strictly convex with a maximum attained at $c=.86$. We have a description for the optimal $c^*$ as follows.

\begin{lemma}
Let $\hat{r}_1$ and $r_2$ be as described above. Then $$c^{*} := {\rm argmax}_c \{\hat{r}_1(c) - r_2(c) \} = \frac{1 - \sqrt{1 - E}}{E},$$ where $$E = \frac{a-b}{a+b}\left(1 - {\rm e}^{-(a+b)/2}\right).$$ 
\end{lemma}

\begin{proof}
To calculate $\hat{r_1}$, we consider the dSBM to have $3$ communities, where we separate the seeds and the rest of the vertices in community 1. Then, Theorem~\ref{PRequations} gives us a system of 3 distributional fixed-point equations. Using those, and calculations similar to the ones we did for Lemma~\ref{firstmomentbound}, we get
\begin{align}
	\hat{r_1} &= (1-c)s \left( \frac{1 - \frac{c \lambda a}{a+b}}{(1 - c\lambda) \left(1 - c \lambda \left(\frac{a-b}{a+b}\right) \right)} - 1 \right)\\ \label{eqhatr1}
	r_2 &= (1-c)s \left( \frac{\frac{c \lambda b}{a+b}}{(1 - c \lambda) \left(1 - c \lambda \frac{a-b}{a+b} \right)}\right). \notag
\end{align}
Now substitute $E = \lambda (a-b)/(a+b)$ to obtain that 
$$\hat{r}_1 - r_2 = \frac{(1-c)s}{1-cE}, $$
and use calculus to compute the optimal $$c^{*} = \frac{1 - \sqrt{1-E}}{E}.$$
\end{proof}

Note that the value $E$ is the second eigenvalue of the matrix on the left hand side of equation $\eqref{matrixform}$. 

\begin{figure}[ht]
\centering
\includegraphics[scale = 0.5]{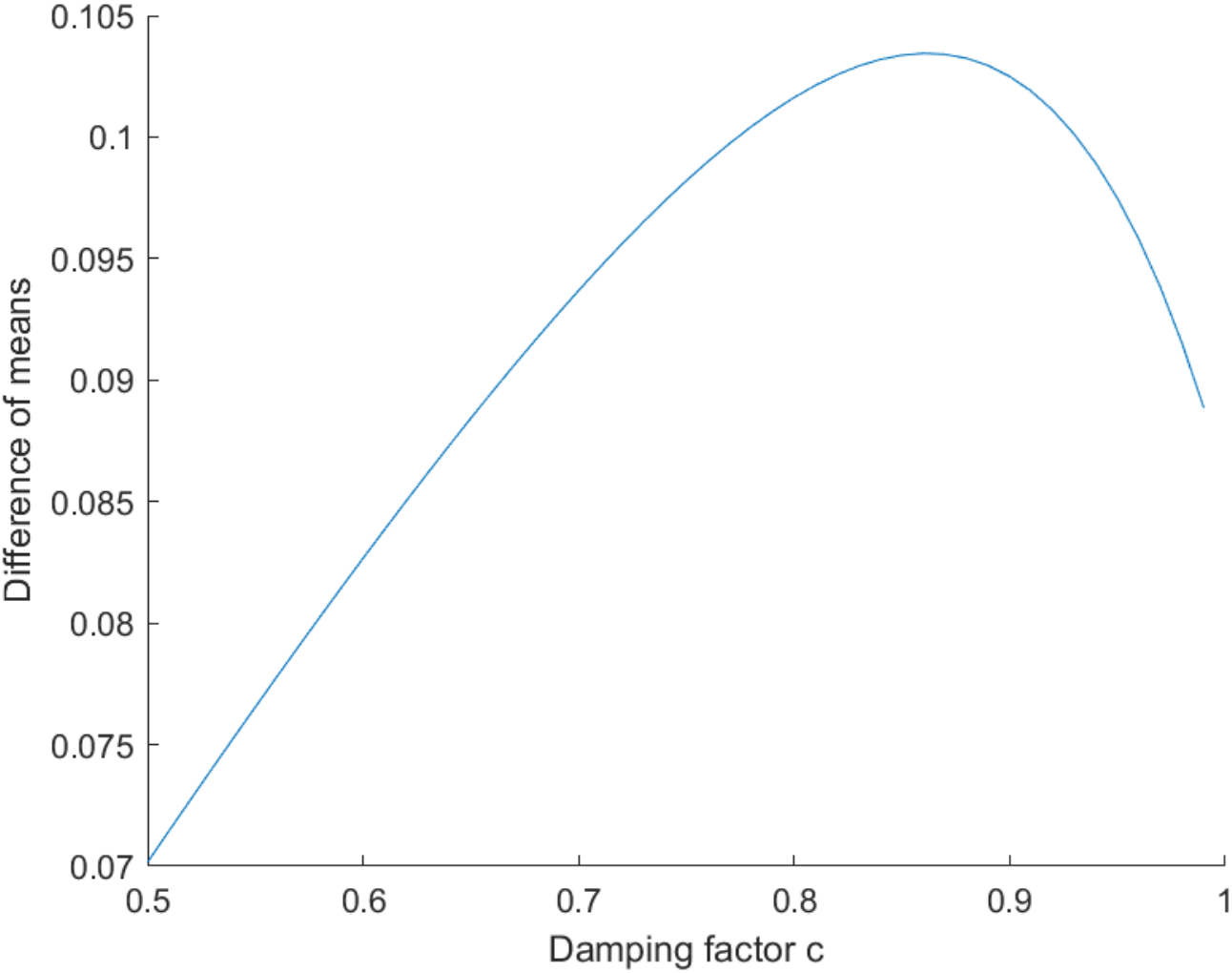}
\caption{Plot of $\hat{r}_1 - r_2$ as c varies from .5 to 1 for a smaller dSBM with $n = 2000$, $a = 100$, $b = 2$ and $s = .15$.}
\label{fig:diffofmeans}
\end{figure}

\section{Remarks and conclusions}
In the sparse regime, we have proposed a cutoff level to identify vertices of community 1 based on their personalized PageRank scores and provided theoretical bounds on the probability of misclassifying a vertex. Our bounds are not tight, and simulations indicate that we might be able to use a lower threshold to further reduce the error (see also Remark \ref{opco}). Another possible threshold option in the case of the symmetric SBM ($\pi_1=\pi_2$) is the median of PageRank scores. We also believe that the proposed method should work for asymmetric dSBMs with $\pi_1 \neq \pi_2$, but the expressions for the mean and variance of PageRank become too complicated to compute clean bounds. Possible future work could 
include trying to show that the $\pi_1$-th quantile of the limiting PageRank distribution is a good threshold in the case $\pi_1 \neq \pi_2$, or trying to find a threshold independent of $\pi$ so that we can recover communities even when we do not have information about their sizes.
Another interesting direction would be to investigate whether the inference can be strengthened if the seed set contains members from both communities and/or the connectivity structure of the subgraph spanned by the seeds is fully or partially known.

%\section{reference notes}
%For citations of references, we prefer the use of square brackets
%and consecutive numbers. Citations using labels or the author/year
%convention are also acceptable. The following bibliography provides
%a sample reference list with entries for journal
%articles~\cite{ref_article1}, an LNCS chapter~\cite{ref_lncs1}, a
%book~\cite{ref_book1}, proceedings without editors~\cite{ref_proc1},
%and a homepage~\cite{ref_url1}. Multiple citations are grouped
%\cite{ref_article1,ref_lncs1,ref_book1},
%\cite{ref_article1,ref_book1,ref_proc1,ref_url1}.

\end{document}